\documentclass[10pt]{amsart} \usepackage{amssymb,amsmath}

\usepackage{graphicx}
\usepackage{pinlabel}
\usepackage{color}

\begin{document}

\newtheorem{theorem}{Theorem}[section] \newtheorem{cor}[theorem]{Corollary} \newtheorem{pro}[theorem]{Property}
\newtheorem{rmk}[theorem]{Remark} \newtheorem{lemma}[theorem]{Lemma} \newtheorem{gen}{Generalization}
\newtheorem{prop}[theorem]{Proposition} \newtheorem{claim}[theorem]{Claim}
\newtheorem{observation}[theorem]{Observation} \newtheorem{notation}[theorem]{Notation}
\newtheorem{conjecture}[theorem]{Conjecture} \newtheorem{defin}[theorem]{Definition}
\newtheorem{defins}[theorem]{Definitions}

\newcommand{\Sssp}{\mbox{$\Sigma_{t_+}$}} \newcommand{\Sssm}{\mbox{$\Sigma_{t_-}$}}
\newcommand{\Ggg}{\mbox{$\Gamma$}} \newcommand{\map}{\mbox{$\rightarrow$}} \newcommand{\bbb}{\mbox{$\beta$}}
\newcommand{\la}{\mbox{$\lambda$}} \newcommand{\aaa}{\mbox{$\alpha$}} \newcommand{\eee}{\mbox{$\epsilon$}}
\newcommand{\Rrr}{\mbox{$\mathcal{R}$}} \newcommand{\lpd}{\mbox{$L^{\mathcal{V}(P,D^*)}$}}
\newcommand{\mcP}{\mbox{$\mathcal{P}$}}
\newcommand{\fpd}{\mbox{$\mathcal{V}(P,D^*)$}} \newcommand{\bdd}{\mbox{$\partial$}}
\newcommand{\Sss}{\mbox{$\Sigma$}} \newcommand{\Li}{\mbox{$L_+^{in}$}} \newcommand{\Lo}{\mbox{$L_+^{out}$}}
\newcommand{\spa}{\hspace{.2cm}}

\title{High Distance Bridge Surfaces} \author{Ryan Blair} \author{Maggy Tomova} \author{Michael Yoshizawa}  \thanks{Research partially supported by an
NSF grant.}

\begin{abstract}
Given integers $b$, $c$, $g$, and $n$, we construct a manifold $M$ containing a $c$-component link $L$ so that there is a bridge surface $\Sigma$ for $(M,L)$ of genus $g$ that intersects $L$ in $2b$ points and has distance at least $n$.  More generally, given two possibly disconnected surfaces $S$ and $S'$, each with some even number (possibly zero) of marked points, and integers $b$, $c$, $g$, and $n$, we construct a compact, orientable manifold $M$ with boundary $S\cup S'$ such that $M$ contains a $c$-component tangle $T$ with a bridge surface $\Sigma$ of genus $g$ that separates $\bdd M$ into $S$ and $S'$, $|T\cap \Sigma|=2b$ and $T$ intersects $S$ and $S'$ exactly in their marked points, and $\Sigma$ has distance at least $n$.
\end{abstract} \maketitle

\section{Introduction}

In recent years, there have been a number of results concerning knots with high distance bridge surfaces. For example, in \cite{BS} Bachman and Schleimer show that any essential surface in the complement of such a knot must have high genus. This result is generalized in \cite{T3}, where it was shown that any other bridge surface of such a knot must also have high genus or a high number of marked points. Several other recent constructions of interesting examples rely on such knots -- in \cite{BT} Blair and Tomova construct knots for which width is not additive, and in \cite{JT} Johnson and Tomova construct examples of knots with two different bridge surfaces that require a high number of stabilizations and perturbations to become isotopic.

We present the first construction of a knot with a high distance bridge surface. In fact, given a collection of (surface, integer)-pairs with some minor restrictions, we show how to construct a (3-manifold, tangle)-pair with the collection of surfaces being the boundary of the manifold and with the prescribed number of boundary points of the tangle on each surface. We can choose our (3-manifold, tangle)-pair to be of arbitrarily high distance and have a given number of components.

Our construction is explicit and purely combinatorial. Many of the ideas in this paper are rooted in \cite{Ev} where Evans constructs 3-manifolds with arbitrarily high distance Heegaard splittings. We have made a number of generalizations and improvements to Evans' arguments. Some of these are necessitated by our more general setting, which allows for marked surfaces (the definition is provided in Section \ref{sec:bridge_surfaces}), and some are introduced to make the arguments more elegant and easier to follow. In particular, we eliminate the language of stacks and partial stacks used by Evans and replace it with a more intuitive description. We also develop tools to control the number of components of the link we are constructing, a challenge that doesn't arise in the set-up considered by Evans.

In most of the paper, we will restrict our attention to a link $L$ in a closed 3-manifold $M$. In the last section of the paper, we will show that our results can be easily extended to tangles in manifolds with boundary as long as each boundary component intersects the tangle in an even number of points (possibly zero).

\vspace{10pt}

\noindent \textbf{Acknowledgements.} We would like to thank the referee for many helpful comments.

\section{Bridge Surfaces and pants decompositions} \label{sec:bridge_surfaces}

\begin{defin} Let $M$ be a closed orientable manifold containing a link $L$. An embedded surface $\Sigma$ in $M$ that is transverse to $L$ is a \emph{bridge surface} for $(M,L)$ if it separates $M$ into two handlebodies $V_1$ and $V_2$ and $L \cap V_i$ is a collection of arcs $A_i$ parallel to $\bdd V_i$.\end{defin}

Arcs parallel to the boundary of a handlebody are naturally associated with certain compressing disks for the manifold obtained by removing a neighborhood of these arcs from the handlebody.

\begin{defin}
Given a boundary parallel arc $\alpha$ properly embedded in a handlebody $V$, the \emph{compressing disk associated to $\alpha$} is the properly embedded disk given by the frontier of the closed regular neighborhood of the disk of parallelism for $\alpha$ in $V$.
\end{defin}

We will often consider closed surfaces in $M$ that intersect $L$ transversely. We will refer to these intersections as marked points to avoid confusion with other types of boundary components the surface may have. Thus,
\begin{defin}
An \emph{$n$-marked} surface is a surface together with $n$ marked points.
\end{defin}

\begin{rmk} In this paper, we will always assume that if the bridge surface is a sphere, it has at least six marked points and if it is a torus, then it has at least two marked points.
\end{rmk}

\begin{defin}
A simple closed curve in an $n$-marked surface is \emph{essential} if it is disjoint from all marked points and it doesn't bound a disk with fewer than two marked points.
\end{defin}

\begin{defin}
Let $\Sigma$ be an $n$-marked surface with marked points $p_1,p_2,...,p_n$.  Two essential curves $x$ and $y$ in $\Sigma$ intersect \emph{efficiently} if they intersect transversely and every bigon they cobound in $\Sigma$ contains at least one marked point.  Note that this is equivalent to $x$ and $y$ having minimal geometric intersection number (which we will denote as $i(x,y)$) in $\Sigma - \{p_1,p_2,...,p_n\}$.
\end{defin}

\begin{defin}
A \emph{pants decomposition} of a closed $2b$-marked surface $\Sigma$ is a collection of essential curves $\mathcal{P}$ such that $\Sigma-\mathcal{P}$ is a collection of open pants and 2-marked open disks.
\end{defin}

We can construct a pants decomposition for a marked surface as follows.  Starting with a closed $2b$-marked surface $\Sigma$, select $b$ disjoint simple closed curves $q_1,q_2,...,q_b$ that bound disjoint disks $Q_1,Q_2,...,Q_b$ in $\Sigma$ such that exactly one pair of points lie in each disk. If we let $\mathcal{R}$ be any pants decomposition of the closure of $\Sigma - \cup_i Q_i$, then $\mathcal{P} = \mathcal{R} \cup (\cup_i q_i)$ is a pants decomposition of $\Sigma$.

\begin{defin}
Given a pants decomposition $\mathcal{P}$ for a closed $2b$-marked surface $\Sigma$ we obtain a handlebody containing a collection of boundary parallel arcs in the following way.  Take $\Sigma \times I$ and attach a 2-handle to $\Sigma \times \{1\}$ along each curve in $\mathcal{P}$.  The boundary of this resulting 3-manifold is composed of $\Sigma \times \{0\}$ and a collection of spheres.  Note that some of these spheres may contain two marked points coinciding with marked points of $\Sigma \times \{1\}$.  Fill each sphere $S_i$ with a 3-ball $B_i$ and denote this 3-manifold as $V_{\mathcal{P}}$, where $\Sigma = \bdd V_{\mathcal{P}}$.  If the sphere $S_j$ contains marked points, embed an arc $\kappa_j$ in $B_j$ that is parallel to $\bdd B_j$ and has endpoints at the marked points.  We can then naturally extend $\kappa_j$ to a properly embedded arc in $V_{\mathcal{P}}$ by taking the union of $\kappa_j$ with the $I$-fibers of $\Sigma \times I$ corresponding to the endpoints of $\kappa_j$.  Let $A_{\mathcal{P}}$ be the collection of these embedded arcs.
% Denote the result $(V_{\mathcal{P}}, A_{\mathcal{P}})$.
\end{defin}

\begin{rmk} Note that the handlebody with boundary parallel arcs constructed above is unique up to isotopies fixing the boundary.
\end{rmk}

\begin{defin}
If $\mathcal{P}$ is a pants decomposition for a $2b$-marked surface $\Sigma$, let $K_\mathcal{P}$ be the collection of all essential curves in $\Sigma$ that bound disks in $V_{\mathcal{P}}$ disjoint from $A_{\mathcal{P}}$.
\end{defin}

Given a handlebody containing a collection of boundary parallel arcs, we want to choose a disjoint collection of compressing disks so that cutting the handlebody along this collection yields components that are either
\begin{itemize}
\item balls with exactly three ``scars" coming from this collection and containing no arcs or
\item balls with exactly one ``scar" and containing a single boundary parallel arc.
\end{itemize}

\begin{defin}\label{def:compcoll}
Let $(V, A)$ be a handlebody $V$ containing a collection of $b$ boundary parallel arcs $A$. A \emph{complete collection of disks} for $(V,A)$ is a collection of disjoint compressing disks $\mathcal{D}$ for the $2b$-marked surface $\bdd V$ that is disjoint from $A$ such that each component of $\bdd V- \bdd \mathcal{D}$ is either an unmarked pair of pants or a $2$-marked disk. A pants decomposition of $\bdd V$ with marked points $\bdd V \cap A$ is any collection of curves that is the boundary of a complete collection of disks. Notice that this is consistent with our general definition of a pants decomposition.
\end{defin}

If $V$ is a genus $g$ handlebody and $A$ is a collection of $b$ boundary parallel arcs, then we can obtain a complete collection of disks for $(V,A)$ in the following way.  Let $\mathcal{E} = \{E_{1}, E_{2}, ..., E_{b}\}$ be a collection of disjoint disks that are compressing disks associated to the arcs of $A$. Cut $V$ along $\mathcal{E}$ to obtain a genus $g$ handlebody $V' \subset V$, where a copy of $\mathcal{E}$ lies in $\partial V'$.  Let $\mathcal{F}$ be a collection of $3g-3$ pairwise disjoint compressing disks for $V'$ such that $\partial \mathcal{F}$ is a pants decomposition of $\partial V'$.  Using collars of $\partial V'$ in $V$, we can isotope $\mathcal{F}$ so that $\partial \mathcal{F}$ lies off of $\mathcal{E} \subset \partial V'$.  Then $\mathcal{F}$ is also a collection of disks in $V$ that is disjoint from $\mathcal{E}$. Consequently, $\bdd V - (\bdd \mathcal{E} \cup \bdd \mathcal{F})$ is composed of either $2$-marked disks or spheres with at least three punctures.  For any $m$-punctured sphere with $m > 3$, we can add a compressing disk $D$ to $\mathcal F$ that is disjoint from all other disks in $\mathcal{E}$ and $\mathcal{F}$ such that $\partial D$ cuts this $m$-punctured sphere into two punctured spheres, each with strictly fewer punctures than $m$.  Repeating this process, we obtain a collection of disks $\mathcal{E} \cup \mathcal{F}$ that is a complete collection of disks for $(V,A)$.

Given a pants decomposition, there are certain arcs in the complementary components that will be of particular importance.

\begin{defin}
Let $\mathcal{P}$ be a pants decomposition of a $2b$-marked surface $\Sigma$ that is the boundary of a handlebody containing boundary parallel arcs.

If $P$ is the closure of a pair of pants component of $(\Sigma - \mathcal{P})$,
\begin{itemize}
\item a \emph{seam} of $P$ is a properly embedded arc in $P$ that has endpoints on different components of $\bdd P$,
\item a \emph{wave} of $P$ is a properly embedded arc in $P$ that has endpoints on the same component of $\bdd P$ and is not $\bdd$-parallel.
\end{itemize}

If $P$ is the closure of a $2$-marked disk component of $(\Sigma - \mathcal{P})$, then a \emph{seam} of $P$ is a properly embedded arc in $P$ that has endpoints on $\partial P$ and is not $\bdd$-parallel.
\end{defin}

Note that waves are only defined in pairs of pants and the definition of a seam is dependent on the type of component of $\Sigma - \mathcal{P}$ we are considering.  In particular, a pair of pants component can contain at most three mutually distinct isotopy classes of seams and a $2$-marked disk has at most one isotopy class of seams.

\begin{defin}
Given a $2b$-marked surface $\Sigma$ with pants decomposition $\mcP$, let $\gamma$ be an embedded curve that intersects $\mcP$ efficiently.  Then $\gamma$ is \emph{$k$-seamed} with respect to $\mcP$ if for each $P$, where $P$ is the closure of a component in $\Sigma - \mathcal{P}$, there are at least $k$ arcs of $\gamma \cap P$ representing each isotopy class of seams of $P$ and the maximum number of mutually distinct isotopy classes of seams of $P$ are represented.
\end{defin}

\begin{defin}
Similarly, a finite collection of disjoint curves $\Gamma = \gamma_1 \cup \gamma_2 \cup ... \cup \gamma_s$ is \emph{$k$-seamed} with respect to $\mcP$ if each $\gamma_i$ intersects $\mcP$ efficiently and for each component $P$, where $P$ is the closure of a component of $\Sigma - \mathcal{P}$, there are at least $k$ arcs of $\Gamma \cap P$ representing each isotopy class of seams of $P$ and the maximum number of mutually distinct isotopy classes of seams of $P$ are represented.
\end{defin}

These definitions imply that each component of $\mcP$ intersects a $k$-seamed curve $\gamma$ (or $k$-seamed collection of curves $\Gamma$) in at least $2k$ points.

\begin{lemma}\label{compdiskwave}
Let $(V, A)$ be a handlebody $V$ containing a collection of $b$ boundary parallel arcs $A$ and let $\Sigma$ be the $2b$-marked boundary of $V$. Additionally, let $\mathcal{D}$ be a complete collection of disks for $(V,A)$ and denote by $\mathcal{P}$ the pants decomposition of $\Sigma$ resulting from the boundary of $\mathcal{D}$. If $\gamma$ is an essential curve in $\Sigma$ that bounds a compressing disk $E$ in $V$ such that $E\cap A=\emptyset$, then $\gamma$ is isotopic to a curve in $\mathcal{P}$ or $\gamma$ contains a wave of a pair of pants component of $\Sigma-\mathcal{P}$.
\end{lemma}

\begin{proof}
Isotope $E$ to intersect $\mathcal{D}$ minimally. If $E\cap \mathcal{D}=\emptyset$ then $\gamma$ is an essential curve contained in a component of $\Sigma-\mathcal{P}$, so $\gamma$ is isotopic to some curve in $\mathcal{P}$. Hence we can assume $E$ intersects $\mathcal{D}$ minimally and $E\cap \mathcal{D}\neq \emptyset$.

Let $\alpha$ be an outermost arc of $E\cap \mathcal{D}$ in $E$ so that $\alpha$ together with an arc $\beta$ in $\partial E$ cobound a disk $E^*$ in $E$ that is disjoint from $\mathcal{D}$ in its interior. Note that, since $E$ has been isotoped to intersect $\mathcal{D}$ minimally, $\alpha$ is an essential arc in $\Sigma-\mathcal{P}$. $E^*$ is contained in some component $N$ of $V-\mathcal{D}$. A component of $V-\mathcal{D}$ is a 3-ball incident to three, not necessarily distinct, disks in $\mathcal{D}$ or a 3-ball containing a single arc of $A$ and incident to a single disk of $\mathcal{D}$. If $N$ is a component of the first type then $\alpha$ is a wave in a pair of pants component of $\Sigma-\mathcal{P}$. If $N$ is a component of the second type then $\partial E^*$ separates the endpoints of the arc of $A$ properly embedded in $N$. This is impossible since $E^*\subset E$ is disjoint from $A$.
\end{proof}

\section{The Dehn Twist Operator}

The following definition of the Dehn twist operator is Definition 3.3 in \cite{Ev} that has been modified to allow for $\Sigma$ to be a marked surface.  The construction presented in \cite{Ev} to obtain the image of the Dehn twist operator also works in our setting.

\begin{defin}
Suppose $X = x_1 \cup x_2 \cup ... \cup x_s$ and $Y = y_1 \cup y_2 \cup ... \cup y_t$ are collections of simple closed curves on a possibly marked surface $\Sigma$ such that $x_i\cap x_j=\emptyset$, $y_i\cap y_j=\emptyset$ and $X\cap y_j \neq \emptyset$ for all $i, j$ and all intersections of $X$ with $Y$ are efficient. An image of a collection $X$ under the \emph{Dehn twist operator along a collection $Y$}, denoted by $\tau_{Y}(X)$, is the union of the images $\{\tau_{Y}(x_1), \tau_{Y}(x_2), ..., \tau_{Y}(x_s)\}$ of $\{x_1, x_2, ..., x_s\}$ where $\tau_Y = \tau_{y_1} \circ \tau_{y_2} \circ ... \circ \tau_{y_t}$.
\end{defin}

\begin{defin}
Let $X=x_1 \cup x_2 \cup ... \cup x_s$ be a collection of pairwise disjoint curves on a possibly marked surface $\Sigma$.  Suppose $A$ is an annulus in $\Sigma$ disjoint from the marked points of $\Sigma$ and with an $I$-fibration such that $A \cap X$ consists of $I$-fibers.  Let $c$ be an essential properly embedded arc in $A$ that intersects each $I$-fiber of $A$ transversely.  Then $c$ has \emph{circling number $k$} with respect to $X$ if it intersects each component of $A \cap X$ at least $k$-times.
\end{defin}

\begin{defin}
Let $X=x_1 \cup x_2 \cup ... \cup x_s$ be a collection of pairwise disjoint curves on a possibly marked surface $\Sigma$ and let $y$ be a simple closed curve on $\Sigma$ with efficient intersection with $X$.  Suppose $A$ is an annular neighborhood of $y$ with an $I$-fibration such that $A \cap X$ are $I$-fibers.  Then a simple closed curve $\gamma$ in $\Sigma$ \emph{k-circles around} $y$ with respect to $X$ if $\gamma$ can be isotoped to have efficient intersection with $X$ and $y$ such that a component of $\gamma \cap A$ has circling number $k$ with respect to $X$ (see Figure \ref{fig:circles}).
\end{defin}

\begin{figure}[h]
\labellist
\footnotesize \hair 2pt
\pinlabel \textcolor{green}{$\gamma$} at 47 108
\pinlabel \textcolor{red}{$X$} at 95 95
\pinlabel \textcolor{blue}{$y$} at 47 78
\pinlabel $A$ at 0 28
\endlabellist
\centering \scalebox{1.5}{\includegraphics{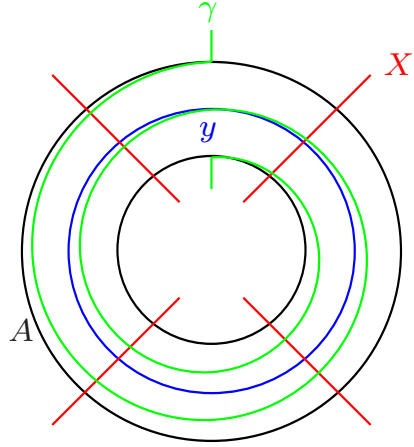}}
\caption{$\gamma$ 2-circles around $y$.}\label{fig:circles}
\end{figure}

The following lemma is a modified version of Lemma 4.3 in Evans \cite{Ev}.

\begin{lemma}\label{almostcontained}

Let $\Sigma$ be an orientable possibly marked surface. Suppose $X = x_1 \cup ... \cup x_s$ and $Y = y_1 \cup...\cup y_t$ are collections of pairwise disjoint essential simple closed curves in $\Sigma$ such that $X$ and $Y$ intersect efficiently and $i(x_i, y_j)\geq 2$ for all $i,j$. Let $\gamma$ and $\gamma'$ be essential simple closed curves in $\Sigma$ such that each meets $Y$ efficiently, $\gamma\cap \gamma'= \emptyset$, $\gamma$ meets $Y$ nontrivially, and $\gamma$ meets $X$ efficiently. If there exists a component $y^1_m$ of $\tau^2_{Y}(X)$ such that $\gamma'$ 1-circles around $y^1_m$ with respect to $X$, then there exists a component $y_l$ of $Y$ such that $\gamma$ 1-circles around $y_l$ with respect to $X$.

\end{lemma}

\begin{proof}
Let $Y^1= y^1_1 \cup... \cup y^1_s=\tau^2_{Y}(X)$, i.e., it is the collection of $s$ disjoint curves obtained by applying the double Dehn twist operator along $Y$ to $X$. One way to describe the result of the double Dehn twist is to define an annulus $A_j$ that is a regular neighborhood of $y_j$ for each $j$ so that all the annuli are pairwise disjoint and $x_i \cap A_j$ is a collection of transverse arcs for each $i$ and $j$. By assumption, for each $i$ and $j$,  $x_i \cap A_j$ contains at least two arcs. Then $y^1_i$ coincides with $x_i$ in the complement of the annuli and each arc $x_i \cap A_j$ is replaced by an arc with the same endpoints but which 2-circles around $y_j$.

Let $y^1_m$ be the component of $Y^1$ so that $\gamma'$ 1-circles around $y^1_m$ with respect to $X$. In other words, there is a subarc of $\gamma'$ that can be isotoped to coincide with a subarc $\beta$ of $y^1_m$ via an ambient isotopy supported in a regular neighborhood of $y^1_m$ and transverse to $X$ so that all intersections between $X$ and $y^1_m$ are contained in $\beta$. We will assume that this isotopy has been performed. Note that as the isotopy is transverse to $X$, it doesn't change the fact that $\gamma$ and $X$ intersect efficiently.

As $\gamma$ meets $Y$ nontrivially, there exists some annulus $A_l$ so that $\gamma \cap A_l \neq \emptyset$. Consider the intersection of $y^1_m$ with $A_l$. Since $2\leq |x_m \cap A_l|=|y^1_m \cap A_l|$ this intersection consists of at least two arcs. At least one of these arcs, $\alpha$, is completely contained in $\beta$, i.e., $\alpha \subset \gamma'$ so $\gamma'$ also 2-circles around $A_l$. As $\gamma$ is disjoint from $\gamma'$ and intersects $y_l$, it must 1-circle around $y_l$ (see Figure \ref{fig:almostcontained}).

\begin{figure}[h]
\labellist
\footnotesize \hair 2pt
\pinlabel \textcolor{magenta}{$\gamma'$} at 44 103
\pinlabel \textcolor{red}{$X$} at 95 95
\pinlabel \textcolor{blue}{$y_l$} at 47 78
\pinlabel $A_l$ at 0 26
\pinlabel \textcolor{green}{$\gamma$} at 76 95
\endlabellist
\centering \scalebox{1.5}{\includegraphics{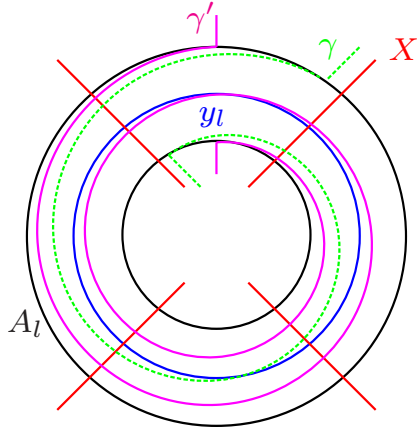}}
\caption{$\gamma$ 1-circles around $y_l$.}\label{fig:almostcontained}
\end{figure}
\end{proof}

The following lemma makes use of the fact that Dehn twisting a pants decomposition of a possibly marked surface $\Sigma$ about a collection of curves in $\Sigma$ still yields a pants decomposition of $\Sigma$.

\begin{lemma}\label{lem:seamed}
If $\mathcal{P}=\{\rho_1,\rho_2,...,\rho_s\}$ is a pants decomposition of an orientable possibly marked surface $\Sigma$ and $\gamma$ is a curve $k$-seamed with respect to $\mathcal{P}$, then
\begin{itemize}
\item $\gamma$ is $k$-seamed with respect to the pants decomposition $\tau^2_{\gamma}(\mathcal{P})$,
\item each curve in $\tau^2_{\gamma}(\mathcal{P})$ is $4k^2$-seamed with respect to $\mathcal{P}$,
\item and each curve in $\mathcal{P}$ is $4k^2$-seamed with respect to $\tau^2_{\gamma}(\mathcal{P})$.
\end{itemize}
\end{lemma}

\begin{proof} Let $y^1_{i}=\tau^2_{\gamma}(\rho_i)$ for each $i$. Observe that the Dehn twist about $\gamma$ can be chosen so that $\gamma$ is fixed pointwise.  So for each subarc of $\gamma$ connecting two components $\rho_{i}$ and $\rho_{j}$ of $\mathcal{P}$, the same subarc of $\gamma$ connects the components $y^1_{i}$ and $y^1_{j}$ of $\tau^2_{\gamma}(\mathcal{P})$.  Therefore every seam contained in $\gamma$ with respect to $\mathcal{P}$ is a seam contained in $\gamma$ with respect to $\tau^2_{\gamma}(\mathcal{P})$ and therefore $\gamma$ is $k$-seamed with respect to $\tau^2_{\gamma}(\mathcal{P})$.

Let $P$ be the closure of a component of $\Sigma-\mathcal{P}$ and let $\rho_{\ell}$ be a component of $\mathcal{P}$ in $\partial P$. Since $\gamma$ is $k$-seamed with respect to $\mathcal{P}$, $\gamma$ meets $\rho_{\ell}$ in at least $2k$ points.  By the definition of the Dehn twist operator, $y^{1}_{\ell}$ is formed by replacing each of the $2k$ distinct subarcs of $\rho_{\ell}$ corresponding to a neighborhood of $\rho_{\ell} \cap \gamma$ in $\rho_{\ell}$ with an arc that 2-circles around $\gamma$. Hence, $y^1_{\ell}$ contains $2k$ copies of every seam of $\mathcal{P}$ for each point in $\rho_{\ell} \cap \gamma$. This is clear for seams in pairs of pants not adjacent to $\rho_{\ell}$. To see it also holds for the seams in a pair of pants adjacent to $\rho_{\ell}$, see Figure \ref{fig:gooddehntwist}. Thus, every curve in $\tau^2_{\gamma}(\mathcal{P})$ is $4k^2$-seamed with respect to $\mathcal{P}$.

We can now apply $\tau^{-2}_{\gamma}$ to the pants decomposition $\tau^2_{\gamma}(\mathcal{P})$.  Since we showed $\gamma$ is $k$-seamed with respect to $\tau^2_{\gamma}(\mathcal{P})$, by the previous argument any curve in $\mathcal{P}$ is $4k^2$-seamed with respect to $\tau^{2}_{\gamma}(\mathcal{P})$.
\end{proof}

 \begin{figure}[h]
 \labellist
 \Huge \hair 2pt
 \pinlabel $\rho_j$ at 140 220
 \pinlabel \textcolor{cyan}{$y^1_j$} at 65 130
 \endlabellist
 \centering \scalebox{.6}{\includegraphics{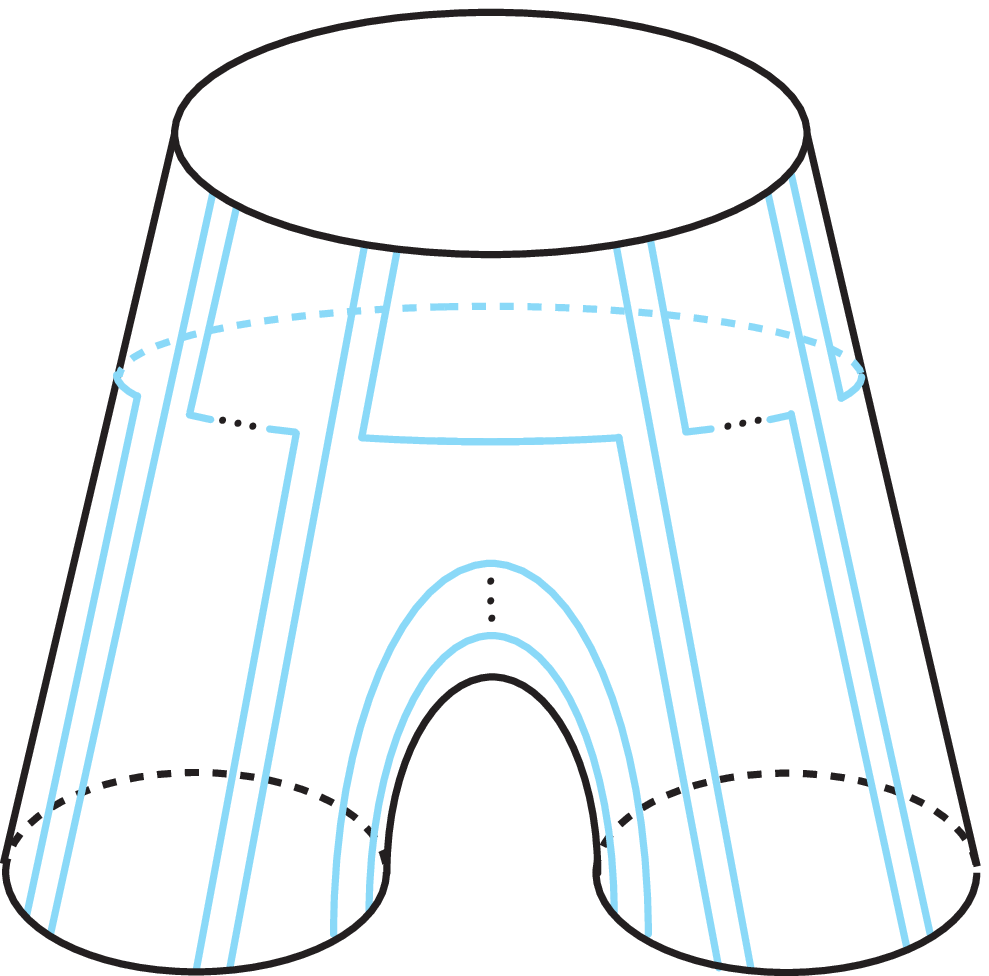}}
 \caption{}\label{fig:gooddehntwist}
 \end{figure}

\begin{lemma}\label{lem:doublestar}
Let $\Sigma$ be an orientable possibly marked surface.  Suppose $\mathcal{P} = \{\rho_{1},...,\rho_{s}\}$ and $\mathcal{R}= \{\omega_{1},...,\omega_{s}\}$ are pants decompositions for $\Sigma$ such that $\mathcal{P}$ and $\mathcal{R}$ are 2-seamed with respect to each other.  Let $\gamma$ and $\gamma'$ be essential simple closed curves in $\Sigma$ such that each meets $\mathcal{R}$ efficiently, $\gamma \cap \gamma' = \emptyset$, and $\gamma$ meets $\mathcal{P}$ efficiently.  If there exists a component $y^1_{m}$ of $\tau^{2}_{\mathcal{R}}(\mathcal{P})$ such that $\gamma'$ 1-circles around $y^1_{m}$ with respect to $\mathcal{P}$ then $\gamma$ must intersect $\mathcal{R}$ nontrivially.
\end{lemma}
\begin{proof}
Let $\rho_{m}$ be the component of $\mathcal{P}$ such that $y^1_{m} = \tau^{2}_{\mathcal{R}}(\rho_{m})$.  By hypothesis, each component of $\mathcal{P}$, and in particular $\rho_{m}$, is 2-seamed with respect to $\mathcal{R}$.  So by Lemma \ref{lem:seamed}, $y^1_{m}$ is also 2-seamed with respect to $\mathcal{R}$.  Moreover, as each component of $\mathcal{R}$ intersects $\mathcal{P}$ at least twice, we can isotope $y^1_{m}$ so that each component of $y^1_{m} - \mathcal{R}$ contains at
least one point of intersection with $\mathcal{P}$ (see Figure \ref{fig:doublestar2}).

\begin{figure}[h]
\labellist
\footnotesize \hair 2pt
\pinlabel $\mathcal{R}$ at 32 88
\pinlabel $\mathcal{R}$ at 92 88
\pinlabel \textcolor{blue}{$y^1_{m}$} at 122 60
\pinlabel $\rho_{m}$ at 0 43
\endlabellist
\centering \scalebox{1.2}{\includegraphics{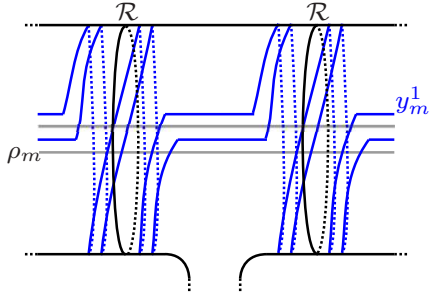}}
\caption{Each component of $y^1_{m} - \mathcal{R}$ intersects $\mathcal{P}$.}\label{fig:doublestar2}
\end{figure}

As $\gamma'$ 1-circles around $y^1_{m}$ with respect to $\mathcal{P}$, there exists a subarc of $\gamma'$ that can be isotoped to coincide with a subarc $\beta$
of $y^1_{m}$ such that $\beta$ contains all points of intersection between $\mathcal{P}$ and $y^1_{m}$.  Since each component of $y^1_m-\mathcal{R}$ contains at least one point of intersection with $\mathcal{P}$, $\beta$ also contains all but at most one point of intersection between $\mathcal{R}$ and $y^1_{m}$.  Thus, $\beta$ (and consequently $\gamma'$) contains at most two fewer seams than $y^1_{m}$ with respect to $\mathcal{R}$.

If we suppose $\gamma$ is disjoint from $\mathcal{R}$, then it must lie parallel to a component of $\mathcal{R}$.  Since $y^1_{m}$ is 2-seamed with respect to $\mathcal{R}$, $\gamma$ must intersect $y^1_{m}$ at least 4 times.  Hence $\gamma$ intersects $\gamma'$ at least twice, a contradiction.
\end{proof}

\section{Main Theorem}

\begin{defin}
The \emph{curve complex} for a marked surface $\Sigma$ is the complex with vertices corresponding to the isotopy classes of essential simple closed curves in $\Sigma$ and a collection of vertices $v_0, v_1, ..., v_k$ defines a $k$-simplex if representatives of the corresponding isotopy classes can be chosen to be pairwise disjoint.  We will denote the curve complex for a surface $\Sigma$ by $\mathcal{C}(\Sigma)$.
\end{defin}

\begin{defin}
Given two collections of essential simple closed curves $X$ and $Y$ in $\Sigma$, the \emph{distance between $X$ and $Y$}, denoted $dist(X,Y)$, is the minimal number of edges in a path in $\mathcal{C}(\Sigma)$ from a vertex corresponding to a curve in $X$ to a vertex corresponding to a curve in $Y$.
\end{defin}

\begin{defin} \label{def:create_manifold}
Let $\Sigma$ be a $2b$-marked surface. If $g(\Sigma)=0$ we require that $b\geq 3$ and if $g(\Sigma)=1$ we require that $b\geq 1$. Let $\mathcal{P}$ and $\mathcal{P}'$ be pants decompositions for $\Sigma$.  Then $V_{\mathcal{P}}$ and $V_{\mathcal{P}'}$ are handlebodies containing a collection of boundary parallel arcs and $V_{\mathcal{P}} \cup_{\Sigma} V_{\mathcal{P}'}$ is a closed orientable manifold $M$ containing a (possibly empty) link $L$ with bridge surface $\Sigma$.  Recall that $K_{\mathcal{P}}$ (resp. $K_{\mathcal{P}'}$) is the set of all essential simple closed curves in the marked surface $\Sigma$ that bound embedded disks in $V_{\mathcal{P}}$ (resp. $V_{\mathcal{P}'}$) that are disjoint from $L$.  Then the {\em distance of the bridge surface $\Sigma$}, denoted $dist(\Sigma,L)$, is equal to $dist(K_{\mathcal{P}},K_{\mathcal{P}'})$.
\end{defin}

We will make extensive use of the following basic facts regarding collections of curves on a surface. The proofs of these facts are left as an exercise to the reader.\\

\textbf{Fact 1:} Given essential collections of curves $X_1$, $X_2$, ..., $X_n$ on a marked surface $\Sigma$ such that all pairwise intersections are efficient except for  $X_i \cap X_j$, there exists an isotopy of $X_i$ fixing all other collections of curves setwise and resulting in all pairwise intersections being efficient.\\

\textbf{Fact 2:} If $X_1$, $X_2$, and $Y$ are collections of curves on a marked surface $\Sigma$ such that $X_1$ intersects $Y$ efficiently, $X_2$ intersects $Y$ efficiently, and $X_1$ is isotopic to $X_2$, then there is an ambient isotopy from $X_1$ to $X_2$ fixing $Y$ setwise.\\

\begin{defin}

Let $D$ be a disk with $2b$ marked points. The \emph{standard} pants decomposition for $D$ is depicted in Figure \ref{fig:standarddisk}. The \emph{basic warped} pants decomposition for $D$ is depicted in Figure \ref{fig:warpdisk}.
\end{defin}

 \begin{figure}[ht]
 \labellist
 \Huge \hair 2pt
 \pinlabel \textcolor{blue}{$\gamma_s$} at -10 85
 \endlabellist
 \centering \scalebox{.8}{\includegraphics{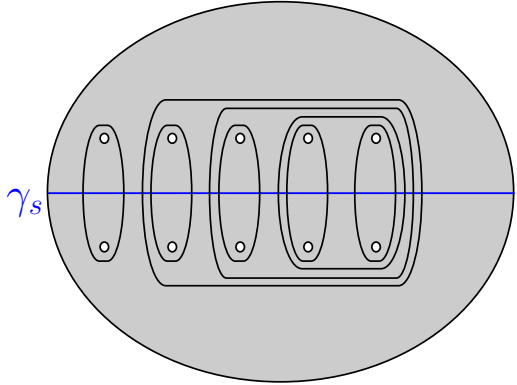}}
 \caption{The standard pants decomposition for a 10-marked disk.}\label{fig:standarddisk}
 \end{figure}

 \begin{figure}[ht]
 \labellist
 \Huge \hair 2pt
 \pinlabel \textcolor{blue}{$\gamma_s$} at -10 85
 \endlabellist
 \centering \scalebox{.8}{\includegraphics{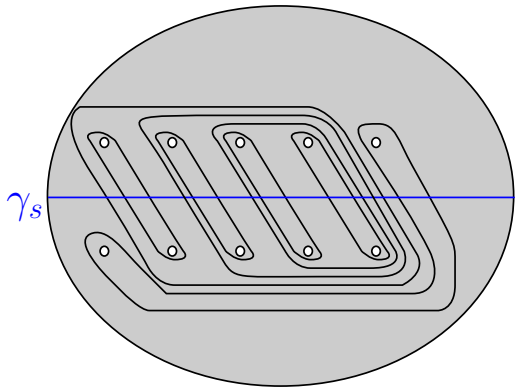}}
 \caption{The basic warped pants decomposition for a 10-marked disk.}\label{fig:warpdisk}
 \end{figure}

\begin{defin}
Let $D$ be a disk with $2b$ marked points and, for $i \geq 2$, let $b_1, b_2, ..., b_i$ be an ordered collection of even integers such that $b_1+ b_2+ ... + b_i=2b$. Let $E$ be a $2i$-marked disk with the standard pants decomposition $\mathcal{Q}$. Number each of the $2$-marked disk components of $E-\mathcal{Q}$ from $1$ to $i$. Replace the $j$th $2$-marked disk component of $E-\mathcal{Q}$ with a $b_j$-marked disk with basic warped pants decomposition $\mathcal{B}_j$.  Then we define the \emph{$\{b_1, b_2, ..., b_i\}$-warped} pants decomposition of $D$ to be $\mathcal{Q}\cup (\cup_{j=1}^{i} \mathcal{B}_j)$. See Figure \ref{fig:bbwarpdisk}.  Note that the \emph{$\{2b\}$-warped} pants decomposition of $D$ will just refer to the basic warped pants decomposition for a $2b$-marked disk.
\end{defin}

 \begin{figure}[h]
 \labellist
 \Huge \hair 2pt
 \pinlabel \textcolor{blue}{$\gamma_s$} at 28 134
 \endlabellist
 \centering \scalebox{.6}{\includegraphics{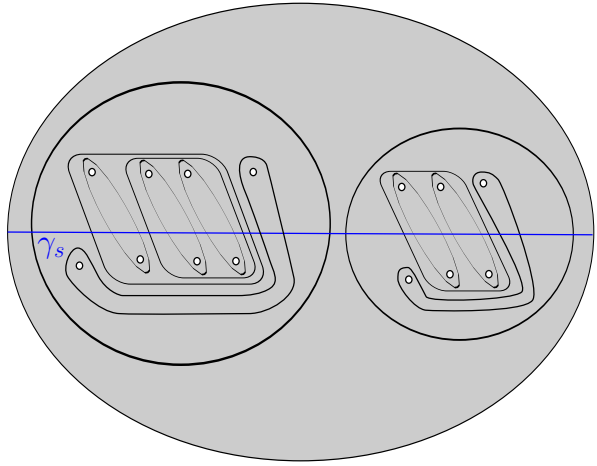}}
 \caption{An $\{8,6\}$-warped pants decomposition for the marked disk.}\label{fig:bbwarpdisk}
 \end{figure}

\begin{defin}
If $S$ is an unmarked, closed surface we define the \emph{seed curve} $\gamma_s \subset S$ in the following ways.

\begin{itemize}
\item If $S$ is a sphere, then $\gamma_s$ is any simple closed curve in $S$.

\item If $S$ is a torus with meridian $m$ and longitude $l$, then $\gamma_s$ is a simple closed curve of slope $\dfrac{1}{2}$ (i.e., $\gamma_s$ is homologous to $m + 2l$).

\item If $S$ has genus $2$ or higher and $\mathcal{R}$ is a pants decomposition for $S$, then $\gamma_s$ is any curve that intersects each pants component of $S - \mathcal{R}$ in exactly 3 mutually non-isotopic seams.
\end{itemize}
\end{defin}

Given a surface $S$ of genus at least 2 with pants decomposition $\mathcal{R}$, we can construct a seed curve $\gamma_s$ in the following way.  For each pants component $P$ determined by $\mathcal{R}$, properly embed 3 disjoint arcs in $P$ so that they are mutually non-isotopic seams in $P$.  This produces a collection of seams embedded in $S$ such that each component of $\mathcal{R}$ contains the endpoints of exactly four seams, with two seams lying on either side.  Then for each component $q \in \mathcal{R}$, we can isotope the four endpoints on $q$ along $q$ so that the endpoint of each seam agrees with the endpoint of exactly one other seam on the other side $q$ and the seams remain pairwise disjoint off of $\mathcal{R}$.  Observe that there are exactly two ways to pair up the four seams adjacent to each component of $\mathcal{R}$.  The result is a properly embedded 1-manifold $\gamma$ in $S$ that contains 3 mutually non-isotopic seams in each pants component of $S - \mathcal{R}$.  If $\gamma$ is disconnected, then since $S$ is connected there exists some $p \in \mathcal{R}$ that intersects two distinct components of $\gamma$.  Surger $\gamma$ along $p$ so that the opposite pairing of arcs at $p$ is obtained.  Note that this strictly reduces the number of components of $\gamma$ by one.  Repeat surgeries if necessary until $\gamma$ is a single component and then set $\gamma_s = \gamma$.

\begin{defin}\label{Xprime}
Let $S$ be a closed surface of genus at least 2, $\mathcal{R}$ a pants decomposition for $S$, and $\alpha \in \mathcal{R}$. Let $\alpha'$ be a curve in $S$ parallel to $\alpha$ and disjoint from $\mathcal{R}$. Let $E$ be the closed regular neighborhood of a point in $\gamma_s$ that lies in the annulus bounded by $\alpha$ and $\alpha'$. Let $\Sigma$ be the surface obtained from $S$ by replacing $E$ with a $2b$-marked disk $D$. Then there are two key pants decompositions for $\Sigma$:

\begin{itemize}
\item The {\em standard pants decomposition of $\Sigma$ induced by $\mathcal{R}$} is $\mathcal{P}=\mathcal{R} \cup \alpha' \cup \mathcal{R}_{D}$ where $\mathcal{R}_{D}$ is the standard pants decomposition for $D$.
\item The {\em $\{b_1, b_2, ..., b_i\}$-warped pants decomposition of $\Sigma$ induced by $\mathcal{R}$} is $\mathcal{P}'=\mathcal{R} \cup \alpha' \cup \mathcal{R}'_{D}$ where $\mathcal{R}'_D$ is a $\{b_1, b_2, ..., b_i\}$-warped pants decomposition for $D$.
\end{itemize}

Let $S$ be a torus with meridian $m$ and longitude $l$. Let $E$ be the closed regular neighborhood of a point in $\gamma_s$ disjoint from $m$. Let $\Sigma$ be the surface obtained from $S$ by replacing $E$ with a $2b$-marked disk $D$. Then there are two key pants decompositions for $\Sigma$:

\begin{itemize}
\item The {\em standard pants decomposition of $\Sigma$ induced by $(m,l)$} is $\mathcal{P}=m \cup \mathcal{R}_{D}$ where $\mathcal{R}_{D}$ is the standard pants decomposition for $D$.
\item The {\em $\{b_1, b_2, ..., b_i\}$-warped pants decomposition of $\Sigma$ induced by $(m,l)$} is $\mathcal{P}'=m \cup \mathcal{R}'_{D}$ where $\mathcal{R}'_D$ is a $\{b_1, b_2, ..., b_i\}$-warped pants decomposition for $D$.
\end{itemize}

Let $S$ be a sphere. Let $E$ be the closed regular neighborhood of a point in $\gamma_s$. Let $\Sigma$ be the surface obtained from $S$ by replacing $E$ with a $2b$-marked disk $D$. Then there are two key pants decompositions for $\Sigma$:

\begin{itemize}
\item The {\em standard pants decomposition of $\Sigma$} is $\mathcal{P}$, obtained from the standard pants decomposition for $D$ by removing curves that become inessential or redundant after inclusion. See Figure \ref{fig:standardsphere}.
\item The {\em $\{b_1, b_2, ..., b_i\}$-warped pants decomposition of $\Sigma$} is $\mathcal{P}'$, obtained from the $\{b_1, b_2, ..., b_i\}$-warped pants decomposition for $D$ by removing curves that become inessential or redundant after inclusion.
\end{itemize}

\end{defin}

 \begin{figure}[h]
 \labellist
 \Huge \hair 2pt
 \pinlabel \textcolor{blue}{$\gamma_s$} at 98 58
 \endlabellist
 \centering \scalebox{.6}{\includegraphics{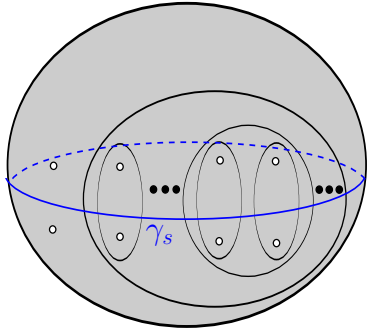}}
 \caption{$\mathcal{P}$ for a marked sphere.}\label{fig:standardsphere}
 \end{figure}

\begin{rmk}\label{rmk:curve}
In each of the pants decompositions described in the previous definition, we insist that $\gamma_s$ persists as a curve in $\Sigma$ that meets $D$ in an arc as depicted in Figure \ref{fig:standarddisk} or Figure \ref{fig:bbwarpdisk}.
\end{rmk}

Note that in our discussion we allow surfaces of genus greater than 1 to not have any marked points. In this case $D=\emptyset$ and the warped and the standard pants decompositions are identical to each other.

The following theorem is our version of Theorem 4.4 in \cite{Ev}.

\begin{theorem}\label{dist}
Let $\Sigma$ be an orientable surface of genus $g$ with $2b$ marked points. If $\Sigma$ is a sphere, assume $b \geq 3$ and if $\Sigma$ is a torus, assume $b \geq 1$. Let $\gamma_s$ be the seed curve for $\Sigma$ with pants decompositions $\mathcal{P}$ and $\mathcal{P}'$ as in Definition \ref{Xprime}. For $n\geq 1$ let

$$Y^0=\tau^2_{\gamma_s}(\mathcal{P})=\{y^0_1,...,y^0_{3g-3+2b}\}$$
$$Y^1=\tau^2_{Y^0}(\mathcal{P})=\{y^1_1,...,y^1_{3g-3+2b}\}$$
$$\vdots$$
$$Y^n=\tau^2_{Y^{n-1}}(\mathcal{P})=\{y^n_1,...,y^n_{3g-3+2b}\}$$

Then $dist(K_{\mathcal{P}'},K_{Y^n})\geq n$ and $dist(K_{\mathcal{P}},K_{Y^n})\geq n$.
\end{theorem}

\begin{proof} In search of a contradiction, assume $dist(K_{\mathcal{P}'},K_{Y^n})=d$ where $d\leq n-1$. Hence, there exist curves $\gamma_0,\gamma_1,...,\gamma_d$ such that $\gamma_0 \in K_{\mathcal{P}'}$ and $\gamma_d \in K_{Y^n}$ and $\gamma_i \cap \gamma_{i+1}=\emptyset$ for each $i$. By Lemma \ref{compdiskwave}, $\gamma_d$ is isotopic to a curve in $Y^n$ or $\gamma_d$ contains a wave of a pair of pants component of $\Sigma-Y^n$. Let $\omega$ be the isotopic copy of $\gamma_d$ that is disjoint from $Y^n$ if we are in the first case and let $\omega$ be the wave subarc of $\gamma_d$ if we are in the second case.

By inductively applying Lemma \ref{lem:seamed}, every curve of $Y^{n-1}$ is 2-seamed with respect to $Y^n$. Hence, $y^{n-1}_{l}\cap \omega \neq \emptyset$ for all $l$.

By construction, $y^n_j$ 2-circles around $y^{n-1}_{l}$ with respect to $\mathcal{P}$ for every $j$ and $l$. Fix an $r$ and $s$, then $y^n_r$ 2-circles around $y^{n-1}_{s}$ with respect to $\mathcal{P}$. As illustrated in Figure \ref{fig:omega}, since the interior of $\omega$ is disjoint from $y^n_r$, $y^n_r$ 2-circles around $y^{n-1}_{s}$ with respect to $\mathcal{P}$, and $y^{n-1}_{s}\cap \omega \neq \emptyset$, it follows that $\gamma_d$ 1-circles around $y^{n-1}_s$. As $\gamma_i \cap \gamma_{i+1}=\emptyset$ for each $i$ and we know by Lemma \ref{lem:doublestar} that the hypotheses of Lemma \ref{almostcontained} hold, we can inductively apply Lemma \ref{almostcontained} to conclude that $\gamma_0$ 1-circles around $y^{n-(d+1)}_{l}$ with respect to $\mathcal{P}$ for some fixed $l$.

\begin{figure}[h]
\labellist
\footnotesize \hair 2pt
% First case
\pinlabel \textcolor{blue}{$y_s^{n-1}$} at 50 102
\pinlabel \textcolor{green}{$\omega$} at 35 75
\pinlabel \textcolor{magenta}{$y_r^{n}$} at 70 57
% Second case
\pinlabel \textcolor{green}{$\omega$} at 40 45
\pinlabel \textcolor{blue}{$y_s^{n-1}$} at 23 18
\pinlabel \textcolor{magenta}{$y_r^{n}$} at 72 3
% Right diagram
\pinlabel \textcolor{magenta}{$y_r^{n}$} at 142 108
\pinlabel \textcolor{green}{$\omega$} at 178 100
\pinlabel \textcolor{red}{$\mathcal{P}$} at 186 95
\pinlabel \textcolor{blue}{$y_s^{n-1}$} at 140 78
\endlabellist
\centering \scalebox{1.5}{\includegraphics{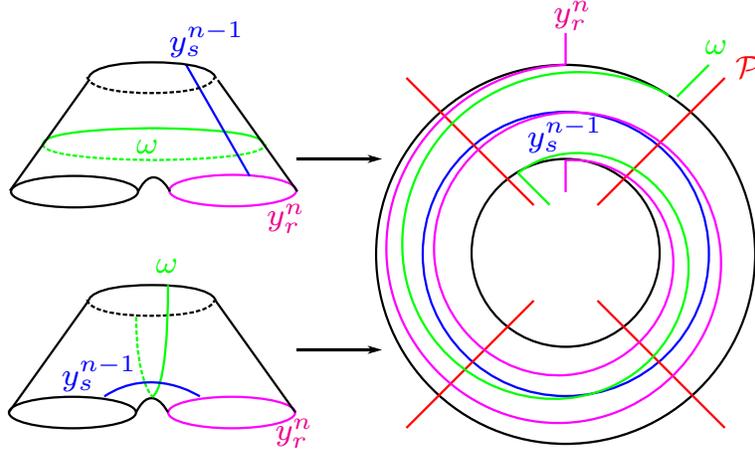}}
\caption{In either case, $\omega$ (and therefore $\gamma_d$) 1-circles around $y^{n-1}_s$.}\label{fig:omega}
\end{figure}

We now describe how to simultaneously establish the following list of criteria.

\begin{enumerate}
\item $\gamma_0$ 1-circles around $y^{n-(d+1)}_{l}$ with respect to $\mathcal{P}$.

\item $y^{n-(d+1)}_{l}$ has efficient intersection with $\mathcal{P}$.

\item $\gamma_0$ has efficient intersection with $\mathcal{P}$.

\item $\gamma_0$ has efficient intersection with $y^{n-(d+1)}_{l}$.

\item $\mathcal{P}$ has efficient intersection with $\mathcal{P}'$.

\item $\gamma_0$ has efficient intersection with $\mathcal{P}'$.

\item $y^{n-(d+1)}_{l}$ has efficient intersection with $\mathcal{P}'$.

\end{enumerate}

Criteria 2, 3, and 4 follow from criteria 1 by the definition of $k$-circling. Criteria 5 follows from Fact 1 applied to $\mathcal{P}$ and $\mathcal{P}'$. Criteria 6 follows from Fact 1 applied to $\mathcal{P}$, $\mathcal{P}'$, and $\gamma_0$. Criteria 7 follows from Fact 1 applied to $\mathcal{P}$, $\mathcal{P}'$, $\gamma_0$, and $y^{n-(d+1)}_{l}$. Notice that in each case we can choose to isotope $\mathcal{P}'$, thus preserving criteria 1. Hence, we can assume that all of the above criteria simultaneously hold.

Since $\gamma_0 \in K_{\mathcal{P}'}$, by Lemma \ref{compdiskwave} $\gamma_0$ is contained in some component of $\Sigma-\mathcal{P}'$ or contains a wave of a pair of pants component of $\mathcal{P}'$. If we are in the first case, $\omega_0$ refers to $\gamma_0$. If we are in the second case, $\omega_0$ refers to some wave of $\mathcal{P}'$ contained in $\gamma_0$. In either case, $\gamma_0$ will intersect any curve that is 1-seamed with respect to $\mathcal{P}'$.

Recall that $D$ refers to the $2b$-marked disk used to define $\mathcal{P}$ and $\mathcal{P'}$ in Definition \ref{Xprime}.

\medskip

\noindent Claim 1: Every curve $y$ in $Y^n$ meets $D$ in at least two arcs isotopic to $\gamma_s \cap D$ and is $2$-seamed with respect to $\mathcal{P}'$ for each $n$.

\noindent Proof: Suppose $y = \tau^2_{\gamma_s}(\rho_i)$ is a curve in $Y^0$. Since $\mathcal{P}$ and $\mathcal{P}'$ are identical outside of $D$, for each pants component $P$ of $\Sigma - \mathcal{P'}$ with interior disjoint from $D$, $y$ contains at least four arcs representing each of 3 mutually non-isotopic classes of seams in $P$ by an argument identical to Lemma \ref{lem:seamed}. If $\rho_i$ is not contained in $D$ then $y\cap D$ consists of four arcs parallel to $\gamma_s \cap D$. If $\rho_i \subset D$, then $y\cap D$ consists of two arcs parallel to $\gamma_s\cap D$ and two additional arcs that are essential in $D$. See Figure \ref{fig:gammax}. Hence $y$ is 2-seamed with respect to $\mathcal{P}'$.

 \begin{figure}[h]
 \labellist
\Huge \hair 2pt
\pinlabel \textcolor{blue}{$y$} at 98 146
\pinlabel $D$ at 193 193
\endlabellist
 \centering \scalebox{.6}{\includegraphics{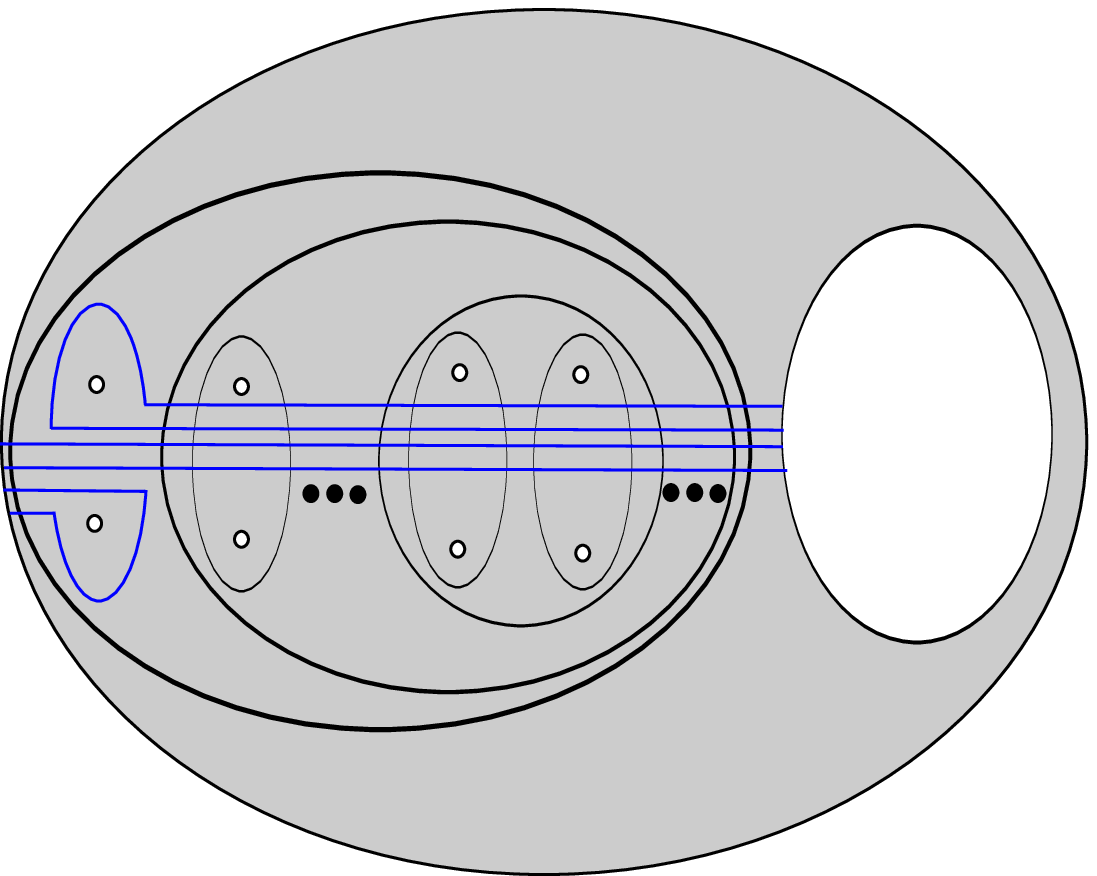}}
 \caption{} \label{fig:gammax}
 \end{figure}

Assume that every curve in $Y^{n-1}$ contains at least 4 copies of every seam represented by $\gamma_s$ with respect to $\mathcal{P}'$ outside of $D$ and $y \cap D$ consists of at least $2$ arcs isotopic to $\gamma_s \cap D$. If $y\in Y^{n}$ then $y = \tau^2_{Y^{n-1}}(\rho_i)$ for some $i$. Since $\mathcal{P}$ and $\mathcal{P}'$ are identical outside of $D$, for each pants component of $\Sigma - \mathcal{P}'$ with interior disjoint from $D$, $y$ contains at least $4 \times 2^2 = 16$ arcs representing each of 3 mutually non-isotopic classes of seams in $P$ by an argument identical to Lemma \ref{lem:seamed}. If $\rho_i$ is not contained in $D$ then $y\cap D$ consists of at least $16$ arcs parallel to $\gamma_s \cap D$. If $\rho_i \subset D$ then $y\cap D$ consists of at least $8$ arcs parallel to $\gamma_s\cap D$ and additional arcs that are essential in $D$. By induction, every curve $y$ in $Y^n$ is $2$-seamed with respect to $\mathcal{P}'$ for each $n$. $\square$

\medskip

In particular, $y^{n-(d+1)}_{l}$ meets $D$ in at least two arcs isotopic to $\gamma_s \cap D$ and is 2-seamed with respect to $\mathcal{P}'$.

Let $A$ be an annular neighborhood of $y^{n-(d+1)}_{l}$. As $\gamma_0$ 1-circles around $y^{n-(d+1)}_{l}$ with respect to $\mathcal{P}$, let $\alpha$ be the arc in $A - \mathcal{P}$ depicted in Figure \ref{fig:ystar} so that the endpoint union of $\alpha$ with a subarc of $\gamma_0$ is a curve $y^*$ which is isotopic in $A$ to $y^{n-(d+1)}_{l}$.

\begin{figure}[h]
 \labellist
 \huge \hair 2pt
 \pinlabel \textcolor{magenta}{$\alpha$} at 92 128
 \pinlabel \textcolor{green}{$\gamma_0$} at 75 167
 \pinlabel \textcolor{red}{$\mathcal{P}$} at 160 155
 \pinlabel \textcolor{magenta}{$y*$} at 286 129
 \pinlabel \textcolor{green}{$\gamma_0$} at 268 167
 \pinlabel \textcolor{red}{$\mathcal{P}$} at 353 155
 \endlabellist\centering \scalebox{.8}{\includegraphics{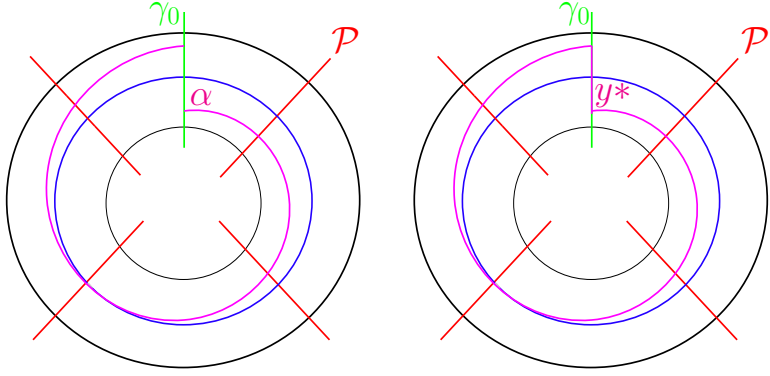}}
 \caption{The curve $y^*$.}\label{fig:ystar}
 \end{figure}

\medskip

\noindent Claim 2: After an isotopy of $\mathcal{P}'$ that preserves criteria 1-7, $y^*$ meets $\mathcal{P}'$ efficiently.

\noindent Proof: Let $D_1$ and $D_2$ be the subdisks of $A$ with boundary in $y^{n-(d+1)}_{l} \cup y^*$ as depicted in Figure \ref{fig:triangles}. Let $\beta$ be an outermost arc of $\mathcal{P}' \cap D_i$ in $D_i$. By criteria 6, $\beta$ cannot have both endpoints in $\gamma_0$. By criteria 7, $\beta$ cannot have both endpoints in $y^{n-(d+1)}_{l}$. If $\beta$ has both endpoints in $\alpha$, eliminate $\beta$ via an isotopy of $\mathcal{P}'$ along the subdisk of $D_i$ with boundary $\beta$ union a subarc of $\alpha$. Since this isotopy is supported in a region disjoint from $y^{n-(d+1)}_{l}\cup \gamma_0 \cup \mathcal{P}$, we can assume criteria 1-7 still hold. If $\beta$ has one boundary point in $\gamma_0$ and one in $\alpha$ then we can eliminate $\beta$ via an isotopy of $\mathcal{P}'$ as in Figure \ref{fig:Xprimeiso}. This isotopy preserves criteria 1-7. We conclude that all arcs of $\mathcal{P}' \cap D_i$ have one endpoint in $y^{n-(d+1)}_{l}$ and the other in either $\alpha$ or $\gamma_0$. Hence, $|y^* \cap \mathcal{P}'|=|y^{n-(d+1)}_{l}\cap \mathcal{P}'|$ and $y^*$ meets $\mathcal{P}'$ efficiently. $\square$

 \begin{figure}[h]
 \labellist
 \huge \hair 2pt
 \pinlabel $D_1$ at 35 80
 \pinlabel $D_2$ at 155 80
 \pinlabel \textcolor{magenta}{$y*$} at 85 30
 \pinlabel \textcolor{green}{$\gamma_0$} at 97 167
 \pinlabel \textcolor{red}{$\mathcal{P}$} at 160 155
 \endlabellist
 \centering \scalebox{.8}{\includegraphics{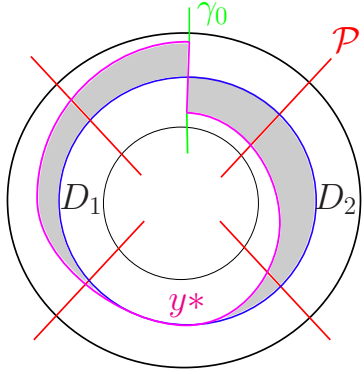}}
 \caption{The disks $D_1$ and $D_2$.}\label{fig:triangles}
 \end{figure}

 \begin{figure}[h]
 \labellist
 \Huge \hair 2pt
 \pinlabel $\mathcal{P}'$ at 98 90
 \pinlabel $D_i$ at 100 30
 \pinlabel $\mathcal{P}'$ at 325 90
 \pinlabel $D_i$ at 330 30
 \endlabellist
 \centering \scalebox{.6}{\includegraphics{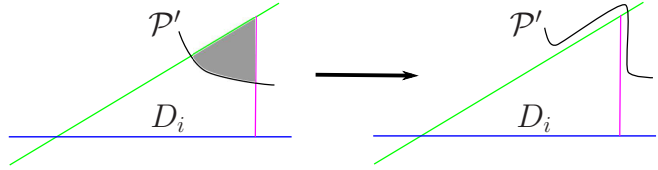}}
 \caption{An isotopy of $\mathcal{P}'$ in a neighborhood of $D_i$.}\label{fig:Xprimeiso}
 \end{figure}

By Fact 2 and the above claim, there is an ambient isotopy from $y^*$ to $y^{n-(d+1)}_{l}$ fixing $\mathcal{P}'$ setwise. Since  $y^{n-(d+1)}_{l}$ meets $D$ in at least two arcs isotopic to $\gamma_s \cap D$ and is 2-seamed with respect to $\mathcal{P}'$, then $y^*$ meets $D$ in at least two arcs isotopic to $\gamma_s \cap D$ and is 2-seamed with respect to $\mathcal{P}'$. Since $\alpha$ is contained in $\Sigma-\mathcal{P}$, then $\gamma_0$ contains at least one copy of every seam of $\mathcal{P}'$ outside of $D$ and at least one copy of $\gamma_s \cap D$. Thus $\gamma_0$ is 1-seamed with respect to $\mathcal{P}'$.  However, we observed $\gamma_0$ contains a subarc $\omega_0$ that will intersect any 1-seamed curve and therefore we have a contradiction.  Hence, $dist(K_{\mathcal{P'}},K_{Y^n}) \geq n$.

By choosing $\mathcal{P}=\mathcal{P'}$, the above argument simplifies considerably and shows $dist(K_{\mathcal{P}},K_{Y^n})\geq n$.
\end{proof}

\begin{cor}
Let $\Sigma$ be a genus $g$ surface with $2b \geq 6$ marked points. Additionally, let $\mathcal{P}$ be a pants decomposition for $\Sigma$. Let $p_1,..,p_k$ be a collection of marked points such that $1<k<2b-1$ and let $n$ be any positive integer. Then there is a curve $\gamma$ that bounds a disk in $\Sigma$ containing exactly the marked points $p_1,..,p_k$ such that $dist(K_{\mathcal{P}}, \gamma)\geq n$.
\end{cor}

\begin{proof}
There are two cases to consider.  In the first case, assume there exists a curve $\rho_i$ in $\mathcal{P}$ such that $\rho_i$ bounds a 2-marked disk in $\Sigma$ containing marked points $p,p' \in \{p_1,...,p_k\}$. For the purposes of applying the construction in the proof of Theorem \ref{dist}, let $\mathcal{P}=\mathcal{P'}$. By the proof of Theorem \ref{dist}, $dist(K_{\mathcal{P}}, y^{n+1}_i)\geq n+1$ where $y^{n+1}_{i}=\tau^{2}_{Y^n}(\rho_i)$. Let $E$ be the disk in $\Sigma$ with boundary $y^{n+1}_{i}$ that contains $p$ and $p'$. Connect each of the remaining $k-2$ marked points in $\{p_1,...,p_k\}$ to $E$ via a collection of $k-2$ disjoint arcs. Let $\gamma$ be the boundary of a regular neighborhood of the union of these arcs and the disk $E$. $\gamma$ is an essential curve in $\Sigma$ bounding a disk containing exactly the marked points $p_1,...,p_k$ and disjoint from $y^{n+1}_{i}$. Hence, $dist(K_{\mathcal{P}}, \gamma)\geq n$.

In the second case, suppose that no component of $\mathcal{P}$ bounds a 2-marked disk in $\Sigma$ containing two marked points in $\{p_1,...,p_k\}$.  So assume instead that $\rho_i$ is a curve in $\mathcal{P}$ that bounds a 2-marked disk in $\Sigma$ containing the marked points $p,p'$ where $p$ is an element of $\{p_1,...,p_k\}$ and $p'$ is not.  For the purposes of applying the construction in the proof of Theorem \ref{dist}, let $\mathcal{P}=\mathcal{P'}$. By the proof of Theorem \ref{dist}, $dist(K_{\mathcal{P}}, y^{n+2}_i)\geq n+2$ where $y^{n+2}_{i}=\tau^{2}_{Y^{n+1}}(\rho_i)$. Let $E$ be the disk in $\Sigma$ with boundary $y^{n+2}_{i}$ and containing $p$ and $p'$. Connect each of the remaining $k-1$ marked points in $\{p_1,...,p_k\}$ to $E$ via a collection of $k-1$ disjoint arcs. Let $\alpha$ be the boundary of a regular neighborhood of the union of these arcs and the disk $E$. Since $k \leq b$, $\alpha$ is an essential curve in $\Sigma$ bounding a disk containing exactly the marked points $p',p_1,...,p_k$. $E$ contains in its interior a disk $E'$ which contains exactly the marked points $ p_1,...,p_k$. Let $\gamma=\partial E'$. Since $\gamma$ is disjoint from $\alpha$ and $\alpha$ is disjoint from $y^{n+2}_{i}$, $dist(K_{\mathcal{P}}, \gamma)\geq n$.
\end{proof}

From the above corollary we can immediately also conclude the following:

\begin{cor}
Let $B$ be a ball containing a $b$-strand tangle $T$ with $b\geq 3$. Let $p_1,..,p_k$ be a collection of points $T \cap \bdd B$ such that $1<k<2b-1$ and let $n$ be any positive integer. Then there is a curve $\gamma$ that bounds a disk in $\bdd B$ containing exactly the marked points $p_1,..,p_k$ such that for every compressing disk $D$ for $\bdd B - T$ contained in $B$, $dist(\bdd D, \gamma)\geq n$.
\end{cor}

\section{High distance Bridge Surfaces}

The definition of a bridge surface for a knot in a closed manifold has a natural generalization to tangles properly embedded in a manifold with boundary:

\begin{defin}
Let $T$ be a tangle properly embedded in a manifold $M$. An embedded closed surface $\Sigma$ transverse to $T$ is a \emph{bridge surface for $(M,T)$} if the closure of the complementary components of $\Sigma$ in $M$ are two compression bodies $V_1$ and $V_2$ with $\Sigma = \bdd_+ V_i$ and $V_i \cap T$ is a collection of arcs that are either boundary parallel to $\Sigma$ or lie vertical (i.e., agree with an $I$-fiber) in $\bdd_- V_i \times I \subset V_i$.  In the case where $T$ is a link, $V_i \cap T$ are all boundary parallel to $\Sigma$ and $|T \cap V_1| = |T \cap V_2|$.  In this case, if $b = |T \cap V_i|$, then $T$ is \emph{$b$-bridged} with respect to $\Sigma$.
\end{defin}

\begin{defin}
Given a bridge surface $\Sigma$ for a tangle $T$ in a manifold $M$ splitting $M$ into compression bodies $V_1$ and $V_2$, let $\mathcal{P}_i$ be the set of simple closed curves in $\Sigma$ that bound compressing disks in $V_i-T$. Then $dist(\Sigma, T)=dist(\mathcal{P}_1, \mathcal{P}_2)$ in $\mathcal{C}(\Sigma)$.
\end{defin}

\begin{cor}\label{cor:highdlink}
Given non-negative integers $b$, $c$, $d$, and $g$ with $c\leq b$ such that if $g=0$, then $b \geq 3$, and if $g=1$, then $b\geq 1$, there exists a closed orientable 3-manifold $M$ containing a $c$-component link $L$ and a bridge surface $\Sigma$ of genus $g$ for $(M,L)$ so that $L$ is $b$-bridge with respect to $\Sigma$ and $dist(\Sigma,L)\geq d$.
\end{cor}

\begin{proof}
Let $\mathcal{P}$ be a standard pants decomposition for a closed, genus $g$, $2b$-marked surface $\Sigma$. Let $\mathcal{P}'$ be the $\{b_1, b_2, ..., b_c\}$-warped pants decomposition for $\Sigma$. Let $Y^d$ be the pants decomposition for $\Sigma$ produced using $\mathcal{P}$ as in Theorem \ref{dist}. Then we can construct a 3-manifold $M$ using $Y^d$ and $\mathcal{P'}$ as in Definition \ref{def:create_manifold}.  Moreover, $\Sigma$ is a bridge surface of a link $L$ such that $\Sigma$ separates $(M,L)$ into $(V_{\mathcal{P'}},A_{\mathcal{P'}})$ and $(V_{Y^d},A_{Y^d})$. By Theorem \ref{dist}, $dist(\Sigma,L)= dist(K_{\mathcal{P'}}, K_{Y^d})\geq d$.

To see that $L$ has $c$ components, pair marked points in $\Sigma$ with respect to some pants decomposition $\mathcal{Q}$ by the rule $x$ is paired with $x'$ if $x$ and $x'$ are contained in a common component of $\Sigma-\mathcal{Q}$. Note that the Dehn twist operator applied to a pants decomposition preserves this pairing. Hence, $\mathcal{P}$ and $Y^d$ induce the same pairing on marked points of $\Sigma$. Since the pairing induced by $Y^d$ (resp. $\mathcal{P'}$) corresponds to which pairs of points of $\Sigma\cap L$ are connected via a subarc of $L$ in $V_{Y^n}$ (resp. $V_{\mathcal{P'}}$), then, by the definition of a $\{b_1, b_2, ..., b_c\}$-warped pants decomposition, $L$ has exactly $c$ components. Compare Figure \ref{fig:standarddisk} and Figure \ref{fig:bbwarpdisk}.
\end{proof}

\begin{cor}\label{cor:hightangle}
Given two possibly disconnected surfaces $S$ and $S'$ with $2p_{S}$ and $2p_{S'}$ marked points respectively so that $g(S)=g_{S}$ and $g(S')=g_{S'}$ and integers $p$, $d$ and $g$ such that $2p\geq max\{2p_{S}, 2p_{S'}\}$, $g \geq max\{g_{S}, g_{S'}\}$, if $g=0$, then $b \geq 3$, and if $g=1$, then $b\geq 1$, there exists an orientable manifold $M$ containing a tangle $T$ so
that there is a bridge surface $\Sigma$ for $(M,T)$ so that all of the following conditions are satisfied:
\begin{enumerate}
\item $\bdd M=S \cup S'$
\item $\Sigma$ separates $\bdd M$ into the sets $S$ and $S'$
\item $genus(\Sigma)=g$
\item $|T\cap \Sigma |=2p$
\item $T$ intersects each surface in $\bdd M$ exactly in the marked points
\item $dist(\Sigma, T)\geq d$.
\end{enumerate}

\end{cor}

\begin{proof}
By Corollary \ref{cor:highdlink} there exists a closed manifold $M$ containing a $c$-component link $L$ so that there is bridge surface $\Sigma$ for $(M,L)$ with $2p$ marked points and with $dist(M,L) \geq d$. Let $V_1$ and $V_2$ be the closures of the two components of $M-\Sigma$ and let $\Gamma_1$ and $\Gamma_2$ be their spines. Recall that a spine $\Gamma_i$ of $V_i$ is a (non-unique) graph embedded such that $V_i$ deformation retracts to it. In particular we may assume that $\Gamma_i$ is homeomorphic to a disjoint collection of loops $\gamma_1,...,\gamma_g$ and edges $t_i$ connecting  $\gamma_i$ and $\gamma_{i+1}$ for $i=1,...,g-1$.

Let $g_1,...,g_r$ be the genera of the components of $S$. Remove from $V_1$ a regular neighborhood of the loops $\gamma_{g_1},...,\gamma_{g_r}$ and all arcs $t_i$
between them except the arcs $t_{g_1}, t_{g_1+g_2},...,t_{g_1+...+g_r}$. The result is a compression body $\tilde{C}_1$ with $\bdd_+\tilde{C}_1=\Sigma$ and
$\bdd_-\tilde{C}_1=S$ so that $\tilde{C}_1 \cap L$ is a collection of boundary parallel arcs. Now replace $p_{S}$ of these boundary parallel arcs with pairs of
vertical (in the product structure of $V_1-\Gamma_1$) arcs with one endpoint coinciding with the endpoints of the replaced arc and the other endpoint in the
marked points of $S$. Let $C_1$ be the new compression body containing the new arcs. Note that $\mathcal{C}(\bdd_+V_1)=\mathcal{C}(\bdd_+C_1)=\mathcal{C}(\Sigma)$ and if a curve in $\Sigma$ bounds a disk in $C_1$ it also bounds a disk in $V_1$.

Perform the analogous operations on $V_2$. The resulting manifold contains a tangle $T$ which has all of the desired properties.
\end{proof}

\begin{rmk}
Note that the previous corollary generalizes so that we can construct $(M,T)$ with the additional flexibility that $T$ has $c$ components for any integer satisfying the inequalities $p_{S}+p_{S'} \leq c \leq min(2p_{S},2p_{S'})+p$. The proof is left as an exercise for the reader.
\end{rmk}

 \end{document}